\documentclass[letterpaper, 10 pt, conference]{ieeeconf}  

\IEEEoverridecommandlockouts                              
\newcommand{\w}{\omega}
\newcommand{\1}{\mathds 1  }

\newcommand{\p}{\partial}

\newcommand{\pH}{port-Hamiltonian}
\newcommand{\G}{\mathcal{G} }
\newcommand{\V}{\mathcal{V} }

\newcommand{\E}{\mathcal{E} }

\usepackage[english]{babel}
\usepackage{amsmath}
\usepackage{amssymb}

\usepackage{graphicx}

\newtheorem{mythm}{Theorem}
\newtheorem{myprop}{Proposition}

\newtheorem{myrem}{Remark}

\newtheorem{myass}{Assumption}

\DeclareMathOperator{\re}{Re}

\DeclareMathOperator{\diag}{diag}
\DeclareMathOperator{\blockdiag}{blockdiag}

\usepackage[utf8]{inputenc}
\usepackage{dsfont}

\title{\LARGE \bf
Optimal power dispatch in networks of high-dimensional models of synchronous machines
}

\author{Tjerk Stegink and Claudio De Persis and Arjan van der Schaft
\thanks{This work is supported by the NWO (Netherlands Organisation for Scientific Research) programme  \emph{Uncertainty Reduction in Smart Energy Systems (URSES)} under the auspices of the project ENBARK.}
\thanks{T.W. Stegink and C. De Persis are with the Engineering and Technology institute Groningen (ENTEG),
        University of Groningen, the Netherlands.
        {\tt\small \{t.w.stegink, c.de.persis\}@rug.nl}}
\thanks{A.J. van der Schaft is with the Johann Bernoulli Institute for Mathematics and Computer Science, University of Groningen, Nijenborgh 9, 9747 AG Groningen,
the Netherlands.
        {\tt\small a.j.van.der.schaft@rug.nl}}%
}

\begin{document}

\maketitle
\thispagestyle{empty}
\pagestyle{empty}

\begin{abstract}
This paper investigates the problem of optimal frequency regulation of multi-machine power networks where each synchronous machine is described by a sixth order model. By analyzing the physical energy stored in the network and the generators, a port-Hamiltonian representation of the multi-machine system is obtained. 
Moreover, it is shown that the open-loop system is passive  with respect to its steady states which implies that passive controllers can be used to control the multi-machine network. As a special case, a distributed consensus based controller is designed that regulates the frequency and minimizes a global quadratic generation cost in the presence of a constant unknown demand. In addition, the proposed controller allows freedom in choosing any desired connected undirected  weighted communication graph. 
\end{abstract}


\section{INTRODUCTION}

The control of power networks has become increasingly challenging over the last decades. As renewable energy sources penetrate the grid, the conventional power plants have more difficulty in keeping the frequency around the nominal value, e.g. 50 Hz, leading to an increased chance of a network failure of even a blackout.  

The current developments require that more advanced models for the power network must be established as the grid is operating more often near its capacity constraints. Considering high-order models of, for example, synchronous machines, that better approximate the reality allows us to establish results on the control and stability of power networks that are more reliable and accurate.

At the same time, incorporating economic considerations in the power grid has become more difficult. As the scale of the grid expands, computing the optimal power production allocation in a centralized manner as conventionally is done is computationally expensive, making distributed control far more desirable compared to centralized control. In addition, often exact knowledge of the power demand is required for computing the optimal power dispatch, which is unrealistic in practical applications. As a result, there is an increased desire for distributed real-time controllers which are able to compensate for the uncertainty of the  demand.

In this paper, we propose an energy-based approach for the modeling, analysis and control of the power grid, both for the physical network as well as for the distributed controller design. Since energy is the main quantity of interest, the port-Hamiltonian framework is a natural approach to deal with the problem. Moreover, the port-Hamiltonian framework lends itself to deal with complex large-scale nonlinear systems like power networks \cite{EJCFiaz},  \cite{stegink2015port}, \cite{LHMNLC}.

The emphasis in the present paper lies on the modeling and control of (networked) synchronous machines as they play an important role in the power network since they are the most flexible and have to compensate for the increased fluctuation of power supply and demand.  However, the full-order model of the synchronous machine as derived in many power engineering books like \cite{anderson1977}, \cite{kundur}, \cite{powsysdynwiley} is difficult to analyze, see e.g. \cite{EJCFiaz} for a \pH\ approach, especially when considering multi-machine networks \cite{cal-tab}, \cite{ortega2005transient}. Moreover,  it is not necessary to consider the full-order model when studying electromechanical dynamics \cite{powsysdynwiley}. 

On the other hand of the spectrum, many of the recent optimal controllers in power grids that deal with optimal power dispatch problems rely on the second-order (non)linear swing equations as the model for the power network \cite{AGC_ACC2014,you2014reverse,zhangpapaautomatica,zhao2015distributedAC}, or the third-order model as e.g. in \cite{trip2016internal}. However, the swing equations are inaccurate and only valid on a specific time scale up to the order of a few seconds so that asymptotic stability results are often invalid for the actual system  \cite{anderson1977}, \cite{kundur}, \cite{powsysdynwiley}. 

Hence, it is appropriate to make simplifying assumptions for the full-order model and to focus on multi-machine models with intermediate complexity which provide a more accurate description of the network compared to the second- and third-order models 
\cite{anderson1977,kundur,powsysdynwiley}. However, for the resulting intermediate-order multi-machine models the stability analysis is often carried out for the linearized system, see   \cite{alv_meng_power_coupl_market,kundur,powsysdynwiley}. Consequently, the stability results are only valid around a specific operating point. 

Our approach is different as the nonlinear nature of the power network is preserved. More specifically, in this paper we consider a nonlinear sixth-order reduced model of the synchronous machine that enables a quite accurate description of the power network while allowing us to perform a rigorous analysis. 

In particular, we show that the \pH\ framework is very convenient when representing the dynamics of the multi-machine network and for the stability analysis. Based on the physical energy stored in the generators and the transmission lines,  a port-Hamiltonian representation of the multi-machine power network can be derived. More specifically, while the system dynamics is complex, the interconnection and damping structure of the \pH\ system is sparse and, importantly, state-independent. 

The latter property implies shifted passivity of the system \cite{phsurvey} which respect to its steady states which allows the usage of passive controllers that steer the system to a desired steady state. As a specific case, we design a distributed real-time controller that regulates the frequency and minimizes the global generation cost without requiring any information about the unknown demand. In addition, the proposed controller design allows us to choose any desired  undirected weighted communication graph as long as the underlying topology is connected. 

The main contribution of this paper is to combine distributed optimal frequency controllers with a high-order nonlinear model  of the power network, which is much more accurate compared to the existing literature, to  prove asymptotic stability to the set of optimal points by using  Lyapunov function based techniques. 





The rest of the paper is organized as follows. In Section \ref{sec:preliminaries} the preliminaries are stated and a sixth order model of a single synchronous machine is given. Next, the multi-machine model is derived in Section \ref{sec:network-equations}. Then the energy functions of the system are derived in Section \ref{sec:energy-analysis}, which are used to represent the multi-machine system in \pH\ form, see Section \ref{sec:port-hamilt-repr}. In Section \ref{sec:minim-gener-costs} the design of the distributed controller is given and asymptotic stability to the set of optimal points is proven. Finally, the conclusions and possibilities for future research are discussed in Section \ref{sec:conclusions}.

\section{PRELIMINARIES}
\label{sec:preliminaries}
Consider a power grid consisting of $ n $ buses. The network is represented by a connected and undirected graph $ \G = (\V, \E) $, where the set of nodes, $ \V = \{1, . . . , n\} $, is the set of buses and the set of edges, $ \E= \{1, . . . , m\} \subset \V \times \V  $, is the set of transmission lines connecting the buses.  The ends of edge $ l \in\E$ are arbitrary labeled with a `+' and a `-', so that the incidence matrix $D$ of the network is given by 
\begin{align}\label{eq:incmatrix}
D_{il}=\begin{cases}
+1 &\text{if $i$ is the positive end of $l$}\\
-1 &\text{if $i$ is the negative end of $l$}\\
0 & \text{otherwise.}
\end{cases}
\end{align}
Each bus represents a synchronous machine and is assumed to
have controllable mechanical power injection and a constant \emph{unknown} power load. 
The dynamics of each synchronous machine  $i\in\V$ is assumed to be  given by \cite{powsysdynwiley} 
\begin{equation}
\begin{aligned}\label{eq:SG6order}
  M_i \dot  \w_i&=P_{mi}-P_{di}-V_{di}I_{di}-V_{qi}I_{qi}\\
\dot \delta_i&= \w_i \\
T_{di}'\dot E_{qi}'&=E_{fi}-E_{qi}'+(X_{di}-X_{di}')I_{di}\\
T_{qi}'\dot E_{di}'&=-E_{di}'-(X_{qi}-X_{qi}')I_{qi} \\
T_{di}''\dot E_{qi}''&=E_{qi}'-E_{qi}''+(X_{di}'-X_{di}'')I_{di} \\
T_{qi}''\dot E_{di}''&=E_{di}'-E_{di}''-(X_{qi}'-X_{qi}'')I_{qi},
\end{aligned} 
\end{equation}
see also Table \ref{tab:varpar}. 
\begin{table}
  \centering
  \begin{tabular}{ll}
    $\delta_i$& rotor angle w.r.t. synchronous reference frame\\
    $\w_i$& frequency deviation\\
    $P_{mi}$& mechanical power injection\\
    $P_{di}$ &  power demand \\
    $M_i$ & moment of inertia\\
    $X_{qi},X_{di}$ & synchronous reactances\\
    $X_{qi}',X_{di}'$ & transient reactances\\
    $X_{di}'',X_{qi}''$ & subtransient reactances\\
    $E_{fi}$ & exciter emf/voltage\\
    $E_{qi}',E_{di}'$ & internal bus transient emfs/voltages\\
    $E_{qi}'',E_{di}''$ & internal bus subtransient emfs/voltages\\
    $V_{qi},E_{di}$ & external bus voltages\\
    $I_{qi},I_{di}$ & generator currents\\
    $T_{qi}',T_{di}'$ & open-loop transient time-scales\\    
    $T_{qi}'',T_{di}''$ & open-loop subtransient time-scales\\
  \end{tabular}
  \caption{Model parameters and variables.}
\label{tab:varpar}
\end{table}
\begin{myass} \label{ass:main}
When using model \eqref{eq:SG6order}, we make the following simplifying assumptions \cite{powsysdynwiley}:  \begin{itemize}
  \item The frequency of each machine is operating around the synchronous frequency.  
  \item The stator winding resistances are zero. 
  \item The excitation voltage $E_{fi}$ is constant for all $i\in\V$.
  \item The \emph{subtransient saliency} is negligible, i.e. $X_{di}''=X_{qi}'', \forall i\in \V$.
  \end{itemize}
The latter assumption is valid for synchronous machines with damper windings in both the $d$ and $q$ axes, which is the case for most synchronous machines \cite{powsysdynwiley}.  
\end{myass}

It is standard in the power system literature to represent the equivalent synchronous machine circuits along the $dq$-axes as in Figure 1,  \cite{kundur}, \cite{powsysdynwiley}. 
Here we use the conventional phasor notation $\overline E_i''= \overline E_{qi}''+\overline E_{di}''= E_{qi}''+jE_{di}''$ where $\overline E_{qi}'':=E_{qi}'', \overline E_{di}'':=jE_{di}''$, and the  phasors $\overline I_i,\overline V_i$ are defined likewise \cite{powsysdynwiley}, \cite{schiffer2015modeling}. Remark that internal voltages $E_{q}',E_d',E_q'',E_d''$ as depicted in Figure 1 are not necessarily at steady state but are  governed by \eqref{eq:SG6order}, where it should be noted that, by definition, the reactances of a round rotor synchronous machine satisfy $X_{di}>X_{di}'>X_{di}''>0, X_{qi}> X_{qi}'>X_{qi}''>0$ for all $i\in \V$ \cite{kundur}, \cite{powsysdynwiley}.
\begin{figure}
\includegraphics{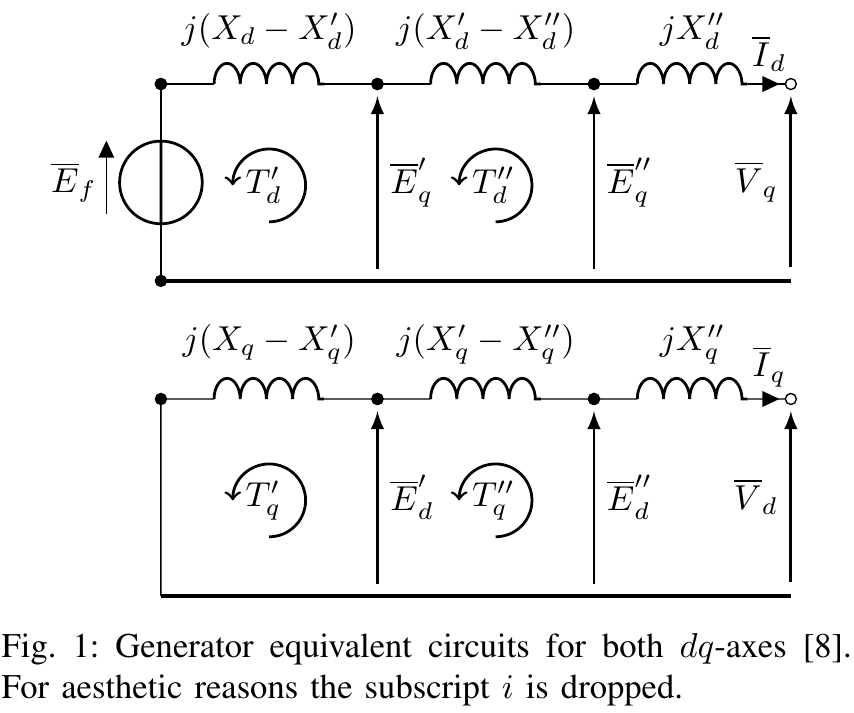}
\label{fig:gencirc}
\end{figure}

By Assumption \ref{ass:main}  the stator winding resistances are negligible so that synchronous machine $i$ can be represented by a subtransient emf behind a subtransient reactance, see Figure 2  \cite{kundur}, \cite{powsysdynwiley}. 
\begin{figure}
\includegraphics{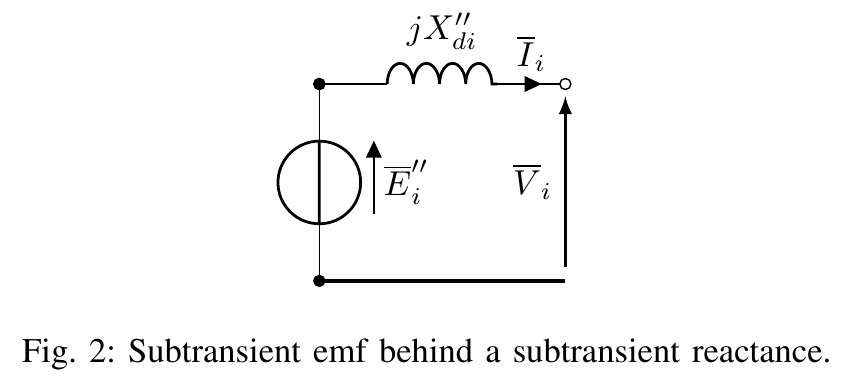}
\label{fig:emfsubtran}
\end{figure}
 As illustrated in this figure, the internal and external voltages are related to each other by \cite{powsysdynwiley}
	\begin{align} \label{eq:EVrelation}
\overline E_i''&=\overline V_i+jX_{di}''\overline I_i, \qquad i\in\V.
	\end{align}

\section{MULTI-MACHINE MODEL}
\label{sec:network-equations}

Consider $n$ synchronous machines which are interconnected by  $RL$-transmission lines and assume that the network is operating at steady state. As the currents and voltages of each synchronous machine is expressed w.r.t. its local $dq$-reference frame, the network equations are written as \cite{schiffer2015modeling}
\begin{align}\label{eq:IEnetworkeq}
  \overline I=\diag(e^{-j\delta_i})\mathcal Y\diag(e^{j\delta_i})\overline E''.
\end{align}
Here the admittance matrix\footnote{Recall that $D$ is the incidence matrix of the network defined by \eqref{eq:incmatrix}.} $\mathcal Y:=D(R+jX)^{-1}D^T$
satisfies  $\mathcal Y_{ik}=-G_{ik}-jB_{ik}$ and $\mathcal Y_{ii}=G_{ii}+j B_{ii}=\sum_{k\in\mathcal N_i}G_{ik}+j\sum_{k\in\mathcal N_i}B_{ik}$ where $G$ denotes the conductance and $B\in\mathbb R_{\leq 0}^{n\times n}$ denotes the susceptance of the network  \cite{schiffer2015modeling}.  In addition, $\mathcal N_i$ denotes the set of neighbors of node $i$.
\begin{myrem}\label{rem:partnetwork}
  As the electrical circuit depicted in Figure 2 
  is in steady state  \eqref{eq:EVrelation}, the reactance
  $X_{di}''$ can also be considered as part of the network (an
  additional inductive line) 
  and is therefore  implicitly  included into the network admittance
  matrix $\mathcal Y$, see also Figure 3. 
\end{myrem}
\begin{figure}
\includegraphics{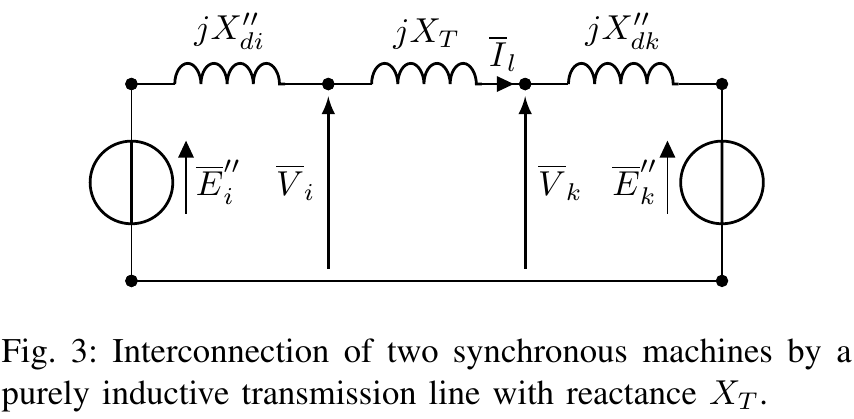}
\label{fig:indline}
\end{figure}

To simplify the analysis further, we assume that the network resistances are negligible so that $G=0$. 
%
%
By equating the real and imaginary part of \eqref{eq:IEnetworkeq} we obtain the following expressions for the $dq$-currents entering generator $i\in \V$:
\begin{equation}
\begin{aligned}\label{eq:currentsmulti}
  I_{di}&=B_{ii}E_{qi}''-\sum_{k\in\mathcal N_i}\left[B_{ik}(E_{dk}''\sin \delta_{ik}+E_{qk}''\cos \delta_{ik})   \right],\\
I_{qi}&=-B_{ii}E_{di}''-\sum_{k\in\mathcal N_i}\left[B_{ik}(E_{qk}''\sin \delta_{ik}-E_{dk}''\cos \delta_{ik})\right],
\end{aligned}
\end{equation}
where $\delta_{ik}:=\delta_i-\delta_k$. By substituting  \eqref{eq:currentsmulti} and \eqref{eq:EVrelation} into \eqref{eq:SG6order}  we obtain after some rewriting a sixth-order multi-machine  model given by equation \eqref{eq:multimach5}, illustrated at the top of the next page.
\begin{figure*}
\begin{equation}
  \begin{aligned}
    M_i\Delta \dot \w_i&=P_{mi}-P_{di}+\sum_{k\in\mathcal N_i}B_{ik}\Big[(E_{di}''E_{dk}'' +E_{qi}''E_{qk}'')\sin \delta_{ik}+(E_{di}''E_{qk}''-E_{qi}''E_{dk}'')\cos \delta_{ik}\Big]\\
    \dot \delta_i&=\Delta \w_i \\
    T_{di}'\dot E_{qi}'&=E_{fi}-E_{qi}'+(X_{di}-X_{di}')(B_{ii}E_{qi}''-\sum_{k\in\mathcal N_i}\left[B_{ik}(E_{dk}''\sin \delta_{ik}+E_{qk}''\cos \delta_{ik})   \right])\\
    T_{qi}'\dot E_{di}'&=-E_{di}'+(X_{qi}-X_{qi}')(B_{ii}E_{di}''-\sum_{k\in\mathcal N_i}\left[B_{ik}(E_{dk}''\cos \delta_{ik}-E_{qk}''\sin \delta_{ik})\right])\\
    T_{di}''\dot E_{qi}''&=E_{qi}'-E_{qi}''+(X_{di}'-X_{di}'')(B_{ii}E_{qi}''-\sum_{k\in\mathcal N_i}\left[B_{ik}(E_{dk}''\sin \delta_{ik}+E_{qk}''\cos \delta_{ik})   \right])\\
    T_{qi}''\dot
    E_{di}''&=E_{di}'-E_{di}''+(X_{qi}'-X_{qi}'')(B_{ii}E_{di}''-\sum_{k\in\mathcal
      N_i}\left[B_{ik}(E_{dk}''\cos \delta_{ik}-E_{qk}''\sin
      \delta_{ik})\right])
  \end{aligned}\label{eq:multimach5}
\end{equation}
\end{figure*}
\begin{myrem}\label{rem:energyconservation}
 Since the transmission lines are purely inductive by assumption,  there are no energy losses in the transmission lines implying that the following energy conservation law holds: $\sum_{i\in\V}P_{ei}=0$ where    $P_{ei}=\re(\overline E_i\overline I_i^*)=E_{di}''I_{di}+E_{qi}''I_{qi}$ is the electrical power produced by synchronous machine $i$.
\end{myrem}

\section{ENERGY FUNCTIONS}
\label{sec:energy-analysis}

When analyzing the stability of the multi-machine system one often searches for a suitable Lyapunov function. A natural starting point is to consider the physical energy as a candidate Lyapunov function. Moreover, when we have an expression for the energy, a port-Hamiltonian representation  of the associated  multi-machine model \eqref{eq:multimach5} can be derived, see Section \ref{sec:port-hamilt-repr}. 

\begin{myrem}\label{rem:scale_ws}
  It is convenient in the definition of the Hamiltonian to multiply the energy stored in the synchronous machine and the transmission lines by the synchronous frequency $\w_s$ since a factor $\w_s^{-1}$ appears in each of the energy functions. As a result, the Hamiltonian has the dimension of power instead of energy. Nevertheless, we still refer to the Hamiltonian as the energy function in the sequel. 
\end{myrem}

In the remainder of this section we will first identify the electrical and mechanical energy stored in each synchronous machine. Next, we identify the   energy stored in the transmission lines.

\subsection{Synchronous Machine}

\subsubsection{Electrical Energy}
Note that, at steady state, the energy (see Remark \ref{rem:scale_ws}) stored in the first two reactances\footnote{In both the $d$- and the $q$-axes.} of generator $i$  as illustrated in Figure 1 is given by  
\begin{equation}
\begin{aligned}\label{eq:ensubtransq}
H_{edi}  &=\frac1{2}\left(\frac{(E_{qi}'-E_{fi})^2}{X_{di}-X_{di}'}+\frac{(E_{qi}'-E_{qi}'')^2}{X_{di}'-X_{di}''} \right)\\
H_{eqi}  &=\frac1{2}\left(\frac{(E_{di}')^2}{X_{qi}-X_{qi}'}+\frac{(E_{di}'-E_{di}'')^2}{X_{qi}'-X_{qi}''} \right).
\end{aligned}
\end{equation}
\begin{myrem}
	The energy stored in the third (subtransient) reactance will be considered as part of the energy stored in the transmission lines, see also Remark \ref{rem:partnetwork} and Section \ref{sec:transmission-lines}. 
\end{myrem}

\subsubsection{Mechanical Energy}
The kinetic energy of synchronous machine $i$ is given by
\begin{align*}
H_{mi}=\frac1{2}M_i\w_i^2=\frac12M_i^{-1}p_i^2,
\end{align*}
where $p_i=M_i\w_i$ is the angular momentum of synchronous machine $i$ with respect to the synchronous rotating reference frame.

\subsection{Inductive Transmission Lines}
\label{sec:transmission-lines}
Consider an interconnection between two SG's with a purely inductive transmission line (with reactance $X_T$) at steady state, see  Figure 3. When expressed in the local $dq$-reference frame of generator $i$, we observe  from  Figure 3 that at steady state one obtains\footnote{The mapping from $dq$-reference frame $k$ to $dq$-reference frame $i$ in the phasor domain is done by multiplication of $e^{-j\delta_{ik}}$ \cite{schiffer2015modeling}.}
  \begin{align}\label{eq:IVindline}
    jX_{l}\overline I_{l}=\overline{E}_i''-e^{-j\delta_{ik}}\overline{E}_k'',
  \end{align}
where the total reactance between the internal buses of generator $i$ and $k$ is given by $X_l:=X_{di}''+X_T+X_{dk}''$. Note that at steady state the modified energy of the inductive transmission line $l$ between nodes $i$ and $k$ is given by $H_{l}=\frac1{2}X_{l}\overline I_{l}^*\overline I_{l},$ which by \eqref{eq:IVindline} can be rewritten as 
\begin{equation}\label{eq:EnIndTransLines}
\begin{aligned}
  H_{l}=-\frac1{2}B_{ik}\Big(&2  \left(E_{di}'' E_{qk}''-E_{dk}'' E_{qi}''\right)\sin \delta _{ik}\\-&2  \left(E_{di}'' E_{dk}''+E_{qi}'' E_{qk}''\right)\cos \delta _{ik}\\
+&E_{di}''^2+E_{dk}''^2+E_{qi}''^2+E_{qk}''^2\Big),
\end{aligned}
\end{equation}
where the line susceptance satisfies $B_{ik}=-\frac1{X_{l}}<0$  \cite{schiffer2015modeling}.

\subsection{Total Energy}
The total physical energy of the multi-machine system is equal to the sum of the individual energy functions:
\begin{align}\label{eq:Hp}
  H_p=\sum_{i\in\V}\left(H_{dei}+H_{qei}+H_{mi}\right)+\sum_{l\in \E}H_{l}.
\end{align}

\section{PORT-HAMILTONIAN REPRESENTATION}
\label{sec:port-hamilt-repr}
Using the energy functions from the previous section, the multi-machine model \eqref{eq:multimach5} can be put into a \pH\ form. To this end, we derive expressions for the gradient of each energy function. \subsection{Transmission Line Energy}
Recall that the energy stored in transmission line  $l$  between internal buses $i$ and $k$ is given by \eqref{eq:EnIndTransLines}. It can be verified that the gradient of the \emph{total} energy $H_L:=\sum_{l\in\E}H_l$ stored in the transmission lines takes the form
\begin{align*}
  \begin{bmatrix}
    \frac{\p H_{L}}{\p \delta_i}\\
    \frac{\p H_{L}}{\p E_{qi}''}\\
    \frac{\p H_{L}}{\p E_{di}''}
  \end{bmatrix}&=                 
  \begin{bmatrix}
  E_{di}''I_{di}+E_{qi}''I_{qi}\\
  I_{di}\\
  I_{qi}
  \end{bmatrix}=
                 \begin{bmatrix}
                   P_{ei}\\
                   I_{di}\\
                   I_{qi}
                 \end{bmatrix},
\end{align*}
where  $I_{di},I_{qi}$ are given by \eqref{eq:currentsmulti}. Here it is used that the self-susceptances satisfy $B_{ii}=\sum_{k\in\mathcal N_i}B_{ik}$ for all $i\in \V$. 

\subsubsection{State transformation}
In the sequel, it is more convenient to consider a different set of variable describing the voltage angle differences. Define for each edge $l\in\E$  $\eta_l:=\delta_{ik}$ where  $i, k$ are respectively the positive and negative ends of $l$. In vector form we obtain $\eta=D^T\delta\in\mathbb R^m$, and observe that this implies
\begin{align*}
 D \frac{\p H_p}{\p \eta}= D \frac{\p H_L}{\p \eta}=P_e.
\end{align*}

\subsection{Electrical Energy SG}
Further, notice that the electrical energy stored in the equivalent circuits along the $d$- and $q$-axis of generator $i$  is given by \eqref{eq:ensubtransq} and satisfies
\begin{align*}
  \begin{bmatrix}
    X_{di}-X_{di}'& X_{di}-X_{di}'\\
    0& X_{di}'-X_{di}''
  \end{bmatrix}
  \begin{bmatrix}
    \frac{\p H_{edi}}{\p E_{qi}'}\\
    \frac{\p H_{edi}}{\p E_{qi}''}
  \end{bmatrix}&=
                 \begin{bmatrix}
                   E_{qi}'-E_{fi}\\
                   E_{qi}''-E_{qi}'
                 \end{bmatrix}\\
\begin{bmatrix}
    X_{qi}-X_{qi}'& X_{qi}-X_{qi}'\\
    0& X_{qi}'-X_{qi}''
  \end{bmatrix}
  \begin{bmatrix}
    \frac{\p H_{eqi}}{\p E_{di}'}\\
    \frac{\p H_{eqi}}{\p E_{di}''}
  \end{bmatrix}&=
                 \begin{bmatrix}
                   E_{di}'\\
                   E_{di}''-E_{di}'
                 \end{bmatrix}.
\end{align*}
By the previous observations, and by aggregating the states, the dynamics of the multi-machine system can now be written in the form \eqref{eq:6multiphall2} where the Hamiltonian is given by \eqref{eq:Hp} and  $\hat X_{di}:=X_{di}-X_{di}', \hat X_{di}':=X_{di}'-X_{di}'', \hat X_d=\diag_{i\in\V}\{\hat X_{di}\}$ and  $\hat X_d',\hat X_{q},\hat X_q'$ are defined likewise. In addition, $T_d'=\diag_{i\in\V}\{T_{di}'\}$ and $T_d',T_q,T_q'$ are defined similarly. 
\begin{figure*}[]
\begin{equation}
\begin{aligned}
\dot x_p&=
\begin{bmatrix}
  \dot p\\
\dot   \eta\\
\dot   E_{q}'\\
\dot   E_{d}'\\
\dot E_{q}''\\
\dot E_{d}''
\end{bmatrix}=
  \begin{bmatrix}
    0&-D&0&0&0&0\\
    D^T&0&0&0&0&0\\
    0&0&-(T_{d}')^{-1}\hat X_{d}&0&-(T_{d}')^{-1}\hat X_{d}&0\\
    0&0&0&-(T_{q}')^{-1}\hat X_{q}&0&-(T_{q}')^{-1}\hat X_{q}\\
    0&0&0&0&-(T_{d}'')^{-1}\hat X_{d}'&0\\
    0&0&0&0&0&-(T_{q}'')^{-1}\hat X_{q}'
  \end{bmatrix}
\nabla H_p          +
g(P_{m}-P_d),\\
y&=g^T
\nabla H_p, \qquad g=  \begin{bmatrix}
     I&0&0&0&0&0
   \end{bmatrix}^T.
\end{aligned}\label{eq:6multiphall2}
\end{equation}
\end{figure*}
Observe that the multi-machine system (\ref{eq:6multiphall2}) is of the form 
\begin{equation}
\begin{aligned}
  \dot x&=(J-R)\nabla H(x)+g u\\
  y&=g^T\nabla H(x)
\end{aligned}\label{eq:consJRpHsys}
\end{equation}
where $J=-J^T$, $ R=R^T$ are respectively the anti-symmetric and symmetric part of the matrix depicted in \eqref{eq:6multiphall2}. Notice that the dissipation matrix of the \emph{electrical part} is positive definite (which implies $R\geq0$) if 
\begin{align*}
\begin{bmatrix}
  2\frac{X_{di}-X_{di}'}{T_{di}'}&\frac{X_{di}-X_{di}'}{T_{di}'}\\
\frac{X_{di}-X_{di}'}{T_{di}'}&2\frac{X_{di}'-X_{di}''}{T_{di}''}
\end{bmatrix}>0, \qquad \forall i\in\V,
\end{align*}
which, by invoking the Schur complement, holds if and only if 
\begin{align}\label{eq:neccT}
  4(X_{di}'-X_{di}'')T_{di}'-(X_{di}-X_{di}')T_{di}''>0,\quad  \forall i\in\V.
\end{align}
 Note that a similar condition holds for the $q$-axis. 
\begin{myprop}
  Suppose that for all $i\in\V$ the following holds:
  \begin{equation}
\begin{aligned}\label{eq:neccTdq}
  4(X_{di}'-X_{di}'')T_{di}'-(X_{di}-X_{di}')T_{di}''&>0\\
  4(X_{qi}'-X_{qi}'')T_{qi}'-(X_{qi}-X_{qi}')T_{qi}''&>0.
\end{aligned}
\end{equation}
Then \eqref{eq:6multiphall2} is a port-Hamiltonian representation of the  multi-machine network \eqref{eq:multimach5}. 
\end{myprop}
It should be stressed that \eqref{eq:neccTdq} is not a restrictive assumption as it holds for a typical generator since $T_{di}''\ll T_{di}', T_{qi}''\ll T_{qi}'$, see also Table 4.2 of \cite{kundur} and Table 4.3 of \cite{powsysdynwiley}.

Because the interconnection and damping structure $J-R$ of \eqref{eq:6multiphall2} is state-independent, the shifted Hamiltonian 
\begin{align}
\bar H(x)=H(x)-(x-\bar x)^T\nabla H(\bar x)-H(\bar x)\label{eq:shiftedH}
\end{align}
acts as a local storage function for proving passivity  in a neighborhood of a steady state $\bar x$ of \eqref{eq:consJRpHsys}, provided that the Hessian of $H$ evaluated at $\bar x$ (denoted as $\nabla^2 H(\bar x)$) is positive definite\footnote{Observe that $\nabla^2 H(x)=\nabla^2 \bar H(x)$ for all $x$.}.
\begin{myprop}\label{prop:incpass}
  Let $\bar u$ be a constant input and suppose there exists a corresponding steady state $\bar x$ to \eqref{eq:consJRpHsys} such that $\nabla^2H(\bar x)>0$.  Then the system \eqref{eq:consJRpHsys} is passive in a neighborhood of $\bar x$ with respect to the shifted external port-variables $\tilde u:=u-\bar u, \tilde y:=y-\bar y$ where $\bar y:=g^T\nabla H(\bar x)$.
\end{myprop}
\begin{proof}
  Define the shifted Hamiltonian by \eqref{eq:shiftedH}, then we obtain 
  \begin{equation}
  \begin{aligned}
    \dot x&=(J-R)\nabla H(x)+gu\\
    &=(J-R)(\nabla \bar H(x)+\nabla H(\bar x))+gu\\
&=(J-R)\nabla \bar H(x)+g(u-\bar u)\\
&=(J-R)\nabla \bar H(x)+g\tilde u\\
\tilde y&=y-\bar y=g^T(\nabla H(x)-\nabla H(\bar x))=g^T\nabla \bar H(x).
  \end{aligned}\label{eq:sysshham}
\end{equation}
As $\nabla^2H(\bar x)>0$ we have that $\bar H(\bar x)=0$ and $\bar H( x)>0$ for all $x\neq \bar x$ in a sufficiently small neighborhood around $\bar x$. Hence, by \eqref{eq:sysshham} the passivity property automatically follows where $\bar H$ acts as a local storage function.
\end{proof}

\section{MINIMIZING GENERATION COSTS}
\label{sec:minim-gener-costs}
The objective is to minimize the total quadratic generation cost  while achieving zero frequency deviation.  By analyzing the steady states of \eqref{eq:multimach5}, it follows that a necessary condition for zero frequency deviation is $\1^TP_{m} = \1^TP_d$, i.e., the total supply must match the total demand. 
Therefore, consider the following convex minimization problem: 
\begin{equation}
\begin{aligned}
  \min_{P_{m}} \ &\frac12P_m^TQP_m\\
  \text{s.t. } & \1^TP_m=\1^TP_d,
\end{aligned}\label{eq:quadmin}
\end{equation}
where  $Q=Q^T>0$ and $P_d$ is a constant \emph{unknown} power load.
\begin{myrem}
	Note that that minimization problem \eqref{eq:quadmin} is easily extended to quadratic cost functions of the form $\frac12P_m^TQP_m+b^TP_m$ for some $b\in\mathbb R^n$. Due to space limitations, this extension is omitted. 
\end{myrem}

As the minimization problem \eqref{eq:quadmin} is convex, it follows that $P_m$ is an optimal solution if and only if the Karush-Kuhn-Tucker conditions are satisfied \cite{convexopt}. Hence, the optimal points of \eqref{eq:quadmin} are characterized by 
\begin{equation}\label{eq:optcond}
\begin{aligned}
P_m^*&=Q^{-1}\1\lambda^*, \qquad 
\lambda^*=\frac{\1^TP_d}{\1^TQ^{-1}\1}
\end{aligned}
\end{equation}
Next, based on the design of \cite{trip2016internal}, consider a distributed controller of the form 
\begin{equation}
\begin{aligned}
  T\dot \theta&=-L_c\theta-Q^{-1}\w\\
  P_m&= Q^{-1}\theta-K\w
\end{aligned}\label{eq:optcontroller}
\end{equation}
where  $T=\diag_{i\in\V}\{T_i\}>0, K=\diag_{i\in\V}\{k_i\}>0$ are  controller parameters. In addition, $L_c$ is the  Laplacian matrix of some connected undirected weighted communication graph.  

The controller  \eqref{eq:optcontroller} consists of three parts. Firstly, the term $-K\w$ corresponds to a primary controller and adds damping into the system. The term $-Q^{-1}\w$ corresponds to secondary control for guaranteeing zero frequency deviation on longer time-scales. Finally, the term $-L_c\theta$ corresponds to tertiary control for achieving optimal production allocation over the network. 

Note that \eqref{eq:optcontroller}  admits the \pH\ representation 
\begin{equation}
\begin{aligned}
  \dot \vartheta&=-L_c\nabla H_c-Q^{-1}\w\\
  P_m&=Q^{-1}\nabla H_c-K\w, \qquad H_c=\frac12\vartheta^TT^{-1}\vartheta,
\end{aligned}\label{eq:contph}
\end{equation}
where $\vartheta:=T\theta$. By interconnecting the controller \eqref{eq:contph}  with \eqref{eq:6multiphall2}, the closed-loop system amounts to
\begin{equation}
\begin{aligned}
  \begin{bmatrix}
    \dot x_p\\
    \dot \vartheta
  \end{bmatrix}
&=
  \begin{bmatrix}
    J-R-R_K&G^T\\
    -G&-L_c
  \end{bmatrix}\nabla H-
  \begin{bmatrix}
    g\\
    0
  \end{bmatrix}
    P_d
\\
G&=
  \begin{bmatrix}
    Q^{-1}&0&0&0&0&0
  \end{bmatrix}\\
\end{aligned}\label{eq:clsys}
\end{equation}
where $J-R$ is given as in \eqref{eq:6multiphall2}, $H:=H_p+H_c$,  and $R_K=\blockdiag(0,K,0,0,0,0)$. Define the set of steady states of \eqref{eq:clsys} by $\Omega$ and observe that any $x:=(x_p,\vartheta)\in\Omega$ satisfies the optimality conditions \eqref{eq:optcond} and $\w=0$. 
\begin{myass}\label{ass:Hessian}
  $\Omega\neq \emptyset$ and there exists  $\bar x\in\Omega$ such that $\nabla^2H(\bar x)>0$. 
\end{myass}
\begin{myrem}
  While the Hessian condition of Assumption \ref{ass:Hessian} is required for proving local asymptotic stability of \eqref{eq:clsys}, guaranteeing that this condition holds can be bothersome. 
However, while we omit the details, it can be shown that $\nabla^2H(\bar x)>0$ if
  \begin{itemize}
  \item the generator reactances are small compared to the transmission line reactances.
  \item the subtransient voltage differences are small.
  \item the rotor angle differences are small.
  \end{itemize}
\end{myrem}
\begin{mythm}\label{thm:main}
  Suppose  $P_d$ is constant and  there exists $\bar x\in\Omega$ such that Assumption \ref{ass:Hessian} is satisfied. Then the trajectories of the closed-loop system \eqref{eq:clsys} initialized in a sufficiently small neighborhood around $\bar x$ converge to the set of optimal points  $\Omega$. 
\end{mythm}
\begin{proof}
  Observe by \eqref{eq:sysshham} that the shifted Hamiltonian defined by \eqref{eq:shiftedH} satisfies
  \begin{align*}
    \dot{\bar H}&=-(\nabla \bar H)^T\blockdiag(R+R_K,L_c)\nabla \bar H\leq 0
  \end{align*}
where equality holds if and only if $\w=0,$ $T^{-1}\vartheta=\theta= \1\theta^*$ for some $\theta^*\in\mathbb R$, and $\nabla_E \bar H(x)=\nabla_EH(x)=0$. Here $\nabla_EH(x)$ is the gradient of $H$ with respect to the internal voltages $E_q',E_d',E_q'',E_d''$. By Assumption \ref{ass:Hessian} there exists a compact neighborhood $\Upsilon$ around $\bar x$ which is forward invariant.  By invoking LaSalle's invariance principle, trajectories initialized in $\Upsilon$ converge to the largest invariant set where $\dot{\bar H}=0$. On this set  $\w,\eta,\theta,E_q',E_d',E_q'',E_d''$ are constant  and, more specifically, $\w=0, \ \theta=\1\lambda^*=\1\frac{\1^TP_d}{\1^TQ^{-1}\1}$   corresponds to an optimal point of \eqref{eq:quadmin} as $P_m=Q^{-1}\1\lambda^*$ where $\lambda^*$ is defined in \eqref{eq:optcond}. We conclude that the trajectories of the closed-loop system \eqref{eq:clsys} initialized in a sufficiently small neighborhood around $\bar x$ converge to the set of optimal points $\Omega$.
\end{proof}
\begin{myrem}
  While by Theorem \ref{thm:main} the trajectories of the closed-loop system \eqref{eq:clsys} converge to the set of optimal points, it may not necessarily converge to a \emph{unique} steady state as the closed-loop system \eqref{eq:clsys} may have  multiple (isolated) steady states. 
\end{myrem}

\section{CONCLUSIONS}
\label{sec:conclusions}
We have shown that a much more advanced multi-machine model than conventionally used can be analyzed using the port-Hamiltonian framework. Based on the energy functions of the system,  a \pH\ representation of the model is obtained. Moreover, the system is proven to be incrementally passive which allows the use of a passive controller that regulates the frequency in an optimal manner, even in the presence of an unknown constant demand. 

The results established in this paper can be extended in many possible ways. Current research has shown that the third, fourth and fifth order model as given in \cite{powsysdynwiley} admit a similar \pH\ structure as \eqref{eq:6multiphall2}. It is expected that the same controller as designed in this paper can also be used in these lower order models. 

While the focus in this paper is about (optimal) frequency regulation, further effort is required to investigate the possibilities of (optimal) voltage control using passive controllers. 
Another extension is to  include transmission line resistances of the network. Finally, one could look at the possibility to extend the results to the case where inverters and frequency dependent loads are included into the network as well.

\bibliographystyle{abbrv} 
\bibliography{mybib}

\clearpage

\end{document}